\documentclass[a4paper,11pt]{article}
\usepackage{amsmath, amsfonts, amscd, amssymb, amsthm, enumerate, xypic}

\def\bfB{\mathbf{B}}
\def\car{\text{car}}
\def\Mat{\text{M}}

\def\End{\text{End}}
\def\id{\text{id}}

\def\defterm{\textbf}
\renewcommand{\L}{\mathbb{L}}
\newcommand{\Sp}{\operatorname{Sp}}

\newcommand{\Ker}{\operatorname{Ker}}

\newcommand{\im}{\operatorname{Im}}
\newcommand{\rg}{\operatorname{rk}}
\renewcommand{\setminus}{\smallsetminus}

\def\K{\mathbb{K}}

\def\N{\mathbb{N}}

\renewcommand{\L}{\mathbb{L}}

\def\calA{\mathcal{A}}
\def\calB{\mathcal{B}}
\def\calC{\mathcal{C}}

\def\lcro{\mathopen{[\![}}
\def\rcro{\mathclose{]\!]}}

\theoremstyle{definition}
\newtheorem{Def}{Definition}
\newtheorem{Not}[Def]{Notation}

\theoremstyle{plain}
\newtheorem{theo}{Theorem}
\newtheorem{prop}[theo]{Proposition}
\newtheorem{cor}[theo]{Corollary}
\newtheorem{lemme}[theo]{Lemma}

\theoremstyle{remark}
\newtheorem{Rems}{Remarks}
\newtheorem{Rem}[Rems]{Remark}

\title{On linear combinations of two idempotent matrices over an arbitrary field}
\author{Cl\'ement de Seguins Pazzis \footnote{Teacher at Lyc\'ee Priv\'e Sainte-Genevi\`eve, 2, rue
de l'\'Ecole des Postes, 78029 Versailles Cedex, FRANCE.}
\footnote{e-mail address: dsp.prof@gmail.com}}

\begin{document}

\maketitle

\begin{abstract}
Given an arbitrary field $\K$ and non-zero scalars $\alpha$ and $\beta$, we give necessary and sufficient conditions for
a matrix $A \in \Mat_n(\K)$ to be a linear combination of two idempotents with coefficients
$\alpha$ and $\beta$. This extends results previously obtained by Hartwig and Putcha
in two ways: the field $\K$ considered here is arbitrary (possibly of characteristic $2$), and
the case $\alpha \neq \pm \beta$ is taken into account.
\end{abstract}



\section{Introduction}

In this article, $\K$ will denote an arbitrary field, $\car (\K)$ its characteristic,
and $n$ a positive integer. We choose an algebraic closure of $\K$ which we denote by $\overline{\K}$.
We let $E$ denote a vector space of dimension $n$ over $\K$, and $\End(E)$ denote the algebra of endomorphisms
of $E$. We choose two scalars $\alpha$ and $\beta$ in $\K^*$.

\vskip 2mm
\noindent
An idempotent matrix of $\Mat_n(\K)$ is a matrix $P$
verifying $P^2=P$, i.e. idempotent matrices represent projectors in finite-dimensional
vector spaces. Of course, any matrix similar to an idempotent is itself an idempotent.

\begin{Def}
Let $\calA$ be a $\K$-algebra. An element $x \in \calA$ will be called an
\textbf{$(\alpha,\beta)$-composite} when there are two idempotents $p$ and $q$
such that $x=\alpha.p+\beta.q$.
\end{Def}

The purpose of this paper is to give necessary and sufficient conditions
on a matrix $A \in \Mat_n(\K)$ to be an $(\alpha,\beta)$-composite, both in terms
of Jordan reduction and elementary factors.
This will generalize the two cases $(\alpha,\beta)=(1,-1)$ and $(\alpha,\beta)=(1,1)$
already discussed in \cite{HP} when the field $\K$ is algebraically closed and $\car(\K) \neq 2$.

\begin{Rems}\label{elremark}
\begin{enumerate}[(i)]
\item Any matrix similar to an $(\alpha,\beta)$-composite is an $(\alpha,\beta)$-composite itself.
\item If $A\in \Mat_n(\K)$ and $B \in \Mat_p(\K)$ are $(\alpha,\beta)$-composites, then
the block diagonal matrix $\begin{bmatrix}
A & 0 \\
0 & B
\end{bmatrix}$ is clearly an $(\alpha,\beta)$-composite itself.
\item The matrix $A \in \Mat_n(\K)$ is an $(\alpha,\beta)$-composite iff $A-\alpha.I_n$ is a $(-\alpha,\beta)$-composite.
\end{enumerate}
\end{Rems}

\begin{Not}
When $A$ is a matrix of $\Mat_n(\K)$, $\lambda \in \overline{\K}$ and $k \in \N^*$, we denote by
$$n_k(A,\lambda):=\dim \Ker (A-\lambda.I_n)^k-\dim \Ker (A-\lambda.I_n)^{k-1},$$
i.e. $n_k(A,\lambda)$ is the number of blocks of size greater or equal to $k$
for the eigenvalue $\lambda$ in the Jordan reduction of $A$
(in particular, it is zero when $\lambda$ is not an eigenvalue of $A$).
We extend this notation to an endomorphism of $E$ provided $\lambda \in \K$.
We also denote by $j_k(A,\lambda)$ the number of size $k$ for the eigenvalue $\lambda$ in the Jordan reduction of $A$.
\end{Not}

\begin{Def}
Two sequences $(u_k)_{k \geq 1}$ and $(v_k)_{k \geq 1}$ are said to be \defterm{intertwined}
when:
$$\forall k \in \N^*, \; v_k \leq u_{k+1} \quad \text{and} \quad u_k \leq v_{k+1.}$$
\end{Def}

\begin{Not}
Let $u \in \End(E)$
and $\Lambda$ be a subset of $\K$.
The minimal polynomial of $u$ splits as
$\mu_u(X)=P(X)\,Q(X)$, where $P$ is a monic polynomial with all its roots in $\Lambda$, and
$Q$ is monic and has no root in $\Lambda$.
We then set
$$u_\Lambda:=u_{|\Ker P(u)} \in \End(\Ker P(u)) \quad \text{and} \quad
u_{-\Lambda}:=u_{|\Ker Q(u)} \in \End(\Ker Q(u)).$$
\end{Not}

\noindent Thus $u_{\Lambda}$ is triangularizable with all eigenvalues in $\Lambda$, whereas
$u_{-\Lambda}$ has no eigenvalue in $\Lambda$.
The kernel decomposition theorem ensures that $u = u_{\Lambda} \oplus u_{-\Lambda}$.
Finally, with $n=\dim E$, the map $u_{\Lambda}$ is an endomorphism of
$\underset{\lambda \in \Lambda}{\bigoplus}\Ker (u-\lambda.\id_E)^n$.

\vskip 4mm
We are now ready to state our main theorems. We will start by generalizations of the Hartwig and Putcha
results on differences of idempotents:

\begin{theo}\label{HPdiff}
Assume $\car(\K) \neq 2$ and let $A \in \Mat_n(\K)$.
Then $A$ is an $(\alpha,-\alpha)$-composite iff all the following conditions hold:
\begin{enumerate}[(i)]
\item The sequences $(n_k(A,\alpha))_{k \geq 1}$ and $(n_k(A,-\alpha))_{k \geq 1}$ are intertwined.
\item $\forall \lambda \in \overline{\K} \setminus \{0,\alpha,-\alpha\}, \; \forall k \in \N^*, \; j_k(A,\lambda)=j_k(A,-\lambda)$.
\end{enumerate}
\end{theo}

\begin{theo}\label{HPdiffendo}
Assume $\car(\K) \neq 2$ and let $u$ be an endomorphism of $E$.
Then $u$ is an $(\alpha,-\alpha)$-composite iff all the following conditions hold:
\begin{enumerate}[(i)]
\item The sequences $\bigl(n_k(u,\alpha)\bigr)_{k \geq 1}$ and $\bigl(n_k(u,-\alpha)\bigr)_{k \geq 1}$ are intertwined.
\item The elementary factors of $u_{-\{0,\alpha,-\alpha\}}$ are all \textbf{even} polynomials
(i.e. polynomials of $X^2$).
\end{enumerate}
\end{theo}

\noindent Using Remark \ref{elremark}.(iii), the previous theorems lead to a characterization of
$(\alpha,\alpha)$-composites when $\car(\K) \neq 2$.

\begin{theo}\label{HPsum}
Assume $\car(\K) \neq 2$ and let $A \in \Mat_n(\K)$.
Then $A$ is an $(\alpha,\alpha)$-composite iff all the following conditions hold:
\begin{enumerate}[(i)]
\item The sequences $(n_k(A,0))_{k \geq 1}$ and $(n_k(A,2\,\alpha))_{k \geq 1}$
are intertwined.
\item $\forall \lambda \in \overline{\K} \setminus \{0,\alpha,2\alpha\}, \; \forall k \in \N^*, \;
j_k(A,\lambda)=j_k(A,2\alpha-\lambda)$.
\end{enumerate}
\end{theo}

\begin{theo}
Assume $\car(\K) \neq 2$ and let $u\in \End(E)$.
Then $u$ is an $(\alpha,\alpha)$-composite iff both of the following conditions hold:
\begin{enumerate}[(i)]
\item The sequences $\bigl(n_k(u,0)\bigr)_{k \geq 1}$ and $\bigl(n_k(u,2\alpha)\bigr)_{k \geq 1}$ are intertwined.
\item The elementary factors of $u_{-\{0,\alpha,2\alpha\}}$ are polynomials of $(X-\alpha)^2$.
\end{enumerate}
\end{theo}

\noindent The case $\car(\K)=2$ works rather differently in terms of Jordan reduction:

\begin{theo}\label{HPsumcar2}
Assume $\car(\K)=2$ and let $A \in \Mat_n(\K)$.
Then $A$ is an $(\alpha,-\alpha)$-composite iff
for every $\lambda \in \overline{\K} \setminus \{0,\alpha\}$,
all blocks in the Jordan reduction of $A$ with respect to $\lambda$ have an even
size.
\end{theo}

\begin{theo}\label{HPsumcar2endo}
Assume $\car(\K)=2$ and let $u \in \End(E)$.
Then $u$ is an $(\alpha,-\alpha)$-composite iff
the elementary factors of $u_{-\{0,\alpha\}}$ are even polynomials.
\end{theo}

\vskip 3mm
\noindent The remaining cases are handled by our two last theorems:

\begin{theo}\label{dSP2LCmat}
Let $A \in \Mat_n(\K)$ and $(\alpha,\beta)\in (\K^*)^2$ such that
$\alpha \neq \pm \beta$. Then $A$ is an $(\alpha,\beta)$-composite iff all the following conditions hold:
\begin{enumerate}[(i)]
\item The sequences $(n_k(A,0))_{k \geq 1}$ and $(n_k(A,\alpha+\beta))_{k \geq 1}$
are intertwined.
\item The sequences $(n_k(A,\alpha))_{k \geq 1}$ and $(n_k(A,\beta))_{k \geq 1}$ are intertwined.
\item $\forall \lambda \in \overline{\K} \setminus \{0,\alpha,\beta,\alpha+\beta\}, \; \forall
k \in \N^*, \; j_k(A,\lambda)=j_k(A,\alpha+\beta-\lambda)$.
\item If in addition $\car(\K)\neq 2$, then
$\forall k \in \N^*, \; j_{2k+1}\bigl(A,\frac{\alpha+\beta}{2}\bigr)=0$.
\end{enumerate}
\end{theo}

\begin{theo}\label{dSP2LCendo}
Let $u\in \End(E)$ and $(\alpha,\beta)\in (\K^*)^2$ such that
$\alpha \neq \pm \beta$. Then $u$ is an $(\alpha,\beta)$-composite iff all the following conditions hold:
\begin{enumerate}[(i)]
\item The sequences $(n_k(u,0))_{k \geq 1}$ and $(n_k(u,\alpha+\beta))_{k \geq 1}$ are intertwined.
\item The sequences $(n_k(u,\alpha))_{k \geq 1}$ and $(n_k(u,\beta))_{k \geq 1}$ are intertwined.
\item The elementary factors of $u_{-\{0,\alpha,\beta,\alpha+\beta\}}$ are polynomials of $(X-\alpha)\,(X-\beta)$.
\end{enumerate}
\end{theo}

\begin{Rem}
A striking consequences of the previous theorems is that being an $(\alpha,\beta)$-composite
is invariant under extension of scalars. More precisely, given a matrix $A \in \Mat_n(\K)$,
an extension $\L$ of $\K$ and non-zero scalars $\alpha$ and $\beta$ in $\K$, the matrix $A$ is an $(\alpha,\beta)$-composite
in $\Mat_n(\K)$ iff it is an $(\alpha,\beta)$-composite in $\Mat_n(\L)$.
\end{Rem}

\vskip 4mm
\noindent The rest of the paper is laid out as follows:
\begin{enumerate}[(i)]
\item In section 3, we show how the odd-labeled theorems can be derived from the even-labeled ones,
e.g. how one can deduce Theorem 1 from Theorem 2.
\item In section 4, we will establish a reduction principle that will show us
that we can limit ourselves to three particular cases for $u \in \End(E)$: the case $u$
has no eigenvalue in $\{0,\alpha,\beta,\alpha+\beta\}$, the case $u$ has all its eigenvalues in
$\{\alpha,\beta\}$ and the case it has all its eigenvalues in $\{0,\alpha+\beta\}$.
\item The case $u$ has no eigenvalue in $\{0,\alpha,\beta,\alpha+\beta\}$
is handled in section 5 by using the reduction to a canonical form and considerations of cyclic matrices.
\item In section 6, we reduce the remaining cases to the sole case $\alpha \neq \beta$ and $u$
has all its eigenvalues in $\{\alpha,\beta\}$, and show how theorems 2, 4, 6 and 8 can be proved
if that case is solved.
\item Finally, in section 7, we solve the case $\alpha \neq \beta$ and $u$ has all its eigenvalues in $\{\alpha,\beta\}$.
\end{enumerate}

\section{Additional notations}

 Similarity of two matrices $A$ and $B$ of $\Mat_n(\K)$ will be written $A\sim B$.
The rank of a matrix $M$ will be written $\rg(M)$, and its spectrum $\Sp(M)$.
Given a list $(A_1,\dots,A_p)$ of square matrices, we will denote by
$$D(A_1,\dots,A_p):=\begin{bmatrix}
A_1 & 0 & & 0 \\
0 & A_2 & & \vdots \\
\vdots & & \ddots & \\
0 & \dots & &  A_p
\end{bmatrix}
$$
the block-diagonal matrix with diagonal blocs $A_1$, \dots, $A_p$.

\begin{Not}
Given a monic polynomial $P=X^n-a_{n-1}X^{n-1}-\dots-a_1X-a_0$, we let
$$C(P):=\begin{bmatrix}
0 & 0 & \dots  & & 0 & a_0 \\
1 & 0 & & & 0 & a_1 \\
0 & 1 & 0 & \dots & 0 & a_2 \\
 &  & \ddots & \ddots & & \vdots\\
\vdots &  & &  1 & 0 & a_{n-2} \\
0 & \dots & \dots & 0 & 1 & a_{n-1}
\end{bmatrix}$$
denote its companion matrix.

\item Given $n \in \N^*$ and $\lambda \in \K$, we set $J_n:=(\delta_{i+1,j})_{1 \leq i,j \leq n}$
i.e.
$$J_n=\begin{bmatrix}
0 & 1 & 0 & \dots & 0 \\
0 & 0 & 1 & \dots & 0 \\
\vdots & & & \ddots & \\
0 & \dots &   & 0 & 1 \\
0 & \dots & &  & 0
\end{bmatrix}$$
and
$$J_\lambda(n):=\lambda.I_n+J_n \quad \text{(the Jordan block of size $n$ associated to $\lambda$).}$$
\end{Not}

\section{Elementary factors vs Jordan reduction}

Derivation of Theorem 1 from Theorem 2 (resp. of Theorem 3 from Theorem 4, resp. of Theorem 5 from Theorem 6, resp.
of Theorem 7 from Theorem 8) can be easily obtained by using the following result and the simple remark
that polynomials of $(X-\alpha)\,(X-\beta)=X^2-(\alpha+\beta)\,X+\alpha\beta$ are polynomials of $X(X-\alpha-\beta)$.

\begin{prop}
Let $A \in \Mat_n(\K)$ and $t \in \K$.
The following conditions are then equivalent:
\begin{enumerate}[(i)]
\item The elementary factors of $M$ are polynomials of $X(X-t)$.
\item For every $\lambda \in \overline{\K}$,
\begin{itemize}
\item if $\lambda \neq t-\lambda$, then $\forall k \in \N^*, \; j_k(A,\lambda)=j_k(A,t-\lambda)$;
\item if $\lambda=t-\lambda$, then $\forall k \in \N, \; j_{2k+1}(A,\lambda)=0$.
\end{itemize}
\end{enumerate}
\end{prop}

\begin{proof}
\begin{itemize}
\item Assume (i). By reduction to an elementary rational canonical form, it suffices to prove condition (ii) when
$A$ is the companion matrix of some polynomial $P=Q(X(X-t))$, with $Q=(Y-\lambda)^r \in \overline{\K}[Y]$ for some
$\lambda \in \K^*$ (remark that
when $Q_1$ and $Q_2$ are mutually prime polynomials, the polynomials $Q_1(X(X-t))$ and $Q_2(X(X-t))$
are mutually prime by the Bezout identity).
\begin{itemize}
\item Assume $X^2-t\,X-\lambda$ has only one root $u$ in
$\overline{\K}$, so it can be written $(X-u)^2$, hence $A=C((X-u)^{2r})$ has only one Jordan block:
this block is even-sized, corresponds to the eigenvalue $u$, and one has $u=t-u$: this proves that $A$ satisfies condition (ii).
\item Assume $X^2-t\,X-\lambda$ has two roots in $\overline{\K}$, let $v$ denote one such root, the other
one being $t-v$. One has then $v \neq t-v$ and
$$A=C\bigl((X-v)^N(X-(t-v))^N\bigr) \sim \begin{bmatrix}
C((X-v)^N) & 0 \\
0 & C((X-t+v)^N)
\end{bmatrix}.$$
In this case, $A$ has only two Jordan blocks, they have the same size and are associated respectively to $v$ and $t-v$, so
$A$ satisfies condition (ii).
\end{itemize}
\vskip 2mm
\item Assume now condition (ii) holds. Let $\mu_A$ denote the minimal polynomial of $A$.
We will first prove that $\mu_A$ is a polynomial of $X(X-t)$.
Since $\delta \mapsto t-\delta$ is an involution, we can split $\Sp(A)$ as a disjoint union
$$\Sp(A)=\calB \cup \calC \cup \calC'$$
where $\calB=\{\delta \in \Sp(A) : \; \delta=t-\delta\}$ and
$\delta \mapsto t-\delta$ is a bijection from $\calC$ to $\calC'$.
For $\delta \in \Sp(A)$, set $r_\delta=\max \{k \in \N^* : \; j_k(A,\delta) \neq 0\}$.
Then the Jordan reduction theorem shows that
$$\mu_A=\underset{\delta \in \Sp(A)}{\prod}(X-\delta)^{r_\delta}.$$
Condition (ii) then entails that $r_\delta=r_{t-\delta}$ for every $\delta \in \calC$
and $r_\delta$ is even when $\delta \in \calB$, hence we may write:
\begin{align*}
\mu_A & =\underset{\delta \in \calB}{\prod}(X-\delta)^{2\,(r_\delta/2)}\,
\underset{\delta \in \calC}{\prod}(X-\delta)^{r_\delta}(X-t+\delta)^{r_\delta} \\
& =\underset{\delta \in \calB}{\prod}(X^2-tX+\delta^2)^{r_\delta/2}
\underset{\delta \in \calC}{\prod}\bigl(X^2-t\,X+\delta(t-\delta)\bigr)^{r_\delta},
\end{align*}
hence $\mu_A$ is a polynomial of $X(X-t)$. \\
However, the theory of elementary factors shows there is a square matrix $B$ such that:
$$A \sim \begin{bmatrix}
B & 0 \\
0 & C(\mu_A)
\end{bmatrix},$$
and it now suffices to show that the elementary factors of $B$ are polynomials of $X(X-t)$.
However $j_k(B,\delta)=j_k(A,\delta)-j_k(C(\mu_A),\delta)$ for every $k \in \N^*$ and $\delta \in \overline{\K}$,
and $A$ and $C(\mu_A)$ satisfy (ii) (for that last matrix, we can use the first part of the proof or simply compute
its Jordan form), so clearly $B$ satisfies (ii). We can thus conclude by downward induction on the size of the matrices.
\end{itemize}
\end{proof}

\section{Reducing the problem}

The first key lemma is a classical one:

\begin{lemme}
Let $P$ and $Q$ be two idempotents in a $\K$-algebra $\calA$.
Then $P$ and $Q$ commute with $(P-Q)^2$.
\end{lemme}

\begin{proof}
Indeed $(P-Q)^2=P+Q-PQ-QP$, so $P\,(P-Q)^2=P-PQP=(P-Q)^2\,P$.
By the same argument, $Q$ commutes with $(Q-P)^2=(P-Q)^2$.
\end{proof}

\begin{cor}\label{fundcor}
Let $P$ and $Q$ be two idempotents in a $\K$-algebra $\calA$, and
set $M:=\alpha.P+\beta.Q$.
Then $P$ and $Q$ commute with $(M-\alpha.I_n)\,(M-\beta.I_n)$.
\end{cor}

\begin{proof}
Indeed, a straightforward computation shows that
$$(M-\alpha.I_n)\,(M-\beta.I_n)=\alpha\,\beta\,\bigl(I_n-(P-Q)^2\bigr).$$
\end{proof}

\noindent Let now $u$ be an endomorphism of $E$ and assume there are idempotents $p$ and $q$ such that
$u=\alpha.p+\beta.q$.

\vskip 4mm
\noindent We decompose the minimal polynomial of $u$ as
$$\mu_u=X^a\,(X-\alpha)^b\,(X-\beta)^c\,(X-\alpha-\beta)^d\,P(X)$$
so that $P$ has no root in $\{0,\alpha,\beta,\alpha+\beta\}$ (in case $\alpha+\beta=0$, we simply take $d=0$).
Since $F:=\Ker P(u)$ is stabilized by $v:=(u-\alpha.\id)\circ (u-\beta.\id)$,
we can define $Q$ as the minimal polynomial
of $v_{|F}$: then $F=\Ker Q(v)$ and $u_{|F}$ has no eigenvalue in $\{0,\alpha,\beta,\alpha+\beta\}$.

\vskip 4mm
\noindent By Corollary \ref{fundcor}, $p$ and $q$ commute with $v$ and therefore stabilize the
three subspaces:
\begin{itemize}
\item $\Ker v^n=\Ker (u-\alpha.\id_E)^n \oplus \Ker (u-\beta.\id_E)^n$;
\item $\Ker (v-\alpha\,\beta.\id_E)^n =\Ker u^n \oplus \Ker (u-(\alpha+\beta).\id_E)^n$;
\item $\Ker Q(v)=\Ker P(u)$.
\end{itemize}
Since $u=\alpha.p+\beta.q$, restricting to those three subspaces shows that
the three endomorphisms $u_{\{\alpha,\beta\}}$, $u_{\{0,\alpha+\beta\}}$ and $u_{-\{0,\alpha,\beta,\alpha+\beta\}}$
are themselves $(\alpha,\beta)$-composites. Using Remark \ref{elremark}.(ii), we deduce the following reduction principle:

\begin{prop}[Reduction principle]\label{redprinc}
Let $u \in \End(E)$. Then $u$ is an $(\alpha,\beta)$-composite iff
both $u_{\{0,\alpha+\beta\}}$, $u_{\{\alpha,\beta\}}$ and $u_{-\{0,\alpha,\beta,\alpha+\beta\}}$
are $(\alpha,\beta)$-composites.
\end{prop}

\noindent We are now reduced to the three special cases that follow:
\begin{itemize}
\item $u$ has no eigenvalue in $\{0,\alpha,\beta,\alpha+\beta\}$;
\item $u$ is triangularizable with all eigenvalues in $\{\alpha,\beta\}$;
\item $u$ is triangularizable with all eigenvalues in $\{0,\alpha+\beta\}$.
\end{itemize}

\section{When no eigenvalue belongs to $\{0,\alpha,\beta,\alpha+\beta\}$}

In this section, $u$ still denotes an endomorphism of $E$.
We assume that $u$ has no eigenvalue in $\{0,\alpha,\beta,\alpha+\beta\}$.  \\
Assume further that there are idempotents $p$ and $q$ such that $u=\alpha.p+\beta.q$.
The assumption on the spectra of $u$ implies that $p$ and $q$ have no common eigenvector, hence
$$\Ker p \cap \Ker q=\Ker p \cap \im q=\im p \cap \Ker q=\im p \cap \im q=\{0\}.$$
As a consequence $\dim \Ker p=\dim \Ker q=\dim \im p=\dim \im q$ and $n$ is even.
It follows that the various kernels and images of $p$ and $q$ all have dimension $m$ for $m:=\dfrac{n}{2}\cdot$
By gluing together a basis of $\Ker q$ and one of $\Ker p$, we obtain a basis $\bfB$ of $E$, together with
square matrices $A \in \Mat_m(\K)$ and $B \in \Mat_m(\K)$
such that
$$M_\bfB(p)=\begin{bmatrix}
I_m & 0 \\
A & 0
\end{bmatrix} \quad \text{and} \quad
M_\bfB(q)=\begin{bmatrix}
0 & B \\
0 & I_m
\end{bmatrix}.$$
Since $\im p \cap \Ker q=\{0\}$, the matrix $A$ is non-singular
By a change of basis, we can reduce the situation to the case
$$M_\bfB(p)=\begin{bmatrix}
I_m & 0 \\
\frac{1}{\alpha}\,I_m & 0
\end{bmatrix} \quad \text{and} \quad
M_\bfB(q)=\begin{bmatrix}
0 & \frac{1}{\beta}\,C \\
0 & I_m
\end{bmatrix}$$
for some $C \in \Mat_m(\K)$, so that
$$M_\bfB(u)=\begin{bmatrix}
\alpha.I_m & C \\
I_m & \beta.I_m
\end{bmatrix}.$$
Conversely, for every $C \in \Mat_m(\K)$, the matrix
$$\begin{bmatrix}
\alpha.I_m & C \\
I_m & \beta.I_m
\end{bmatrix}
=\alpha.
\begin{bmatrix}
I_m & 0 \\
\frac{1}{\alpha}\,I_m & 0
\end{bmatrix}
+\beta.\begin{bmatrix}
0 & \frac{1}{\beta}\,C \\
0 & I_m
\end{bmatrix}
$$
is an $(\alpha,\beta)$-composite.

\vskip 2mm
\noindent We have thus proven that, for every $M \in \Mat_n(\K)$ with no eigenvalue in $\{0,\alpha,\beta,\alpha+\beta\}$,
the following conditions are equivalent:
\begin{enumerate}[(i)]
\item $M$ is an $(\alpha,\beta)$-composite;
\item The integer $n$ is even and there exists $C\in \Mat_{n/2}(\K)$ such that
$$M \sim \begin{bmatrix}
\alpha.I_{n/2} & C \\
I_{n/2} & \beta.I_{n/2}
\end{bmatrix}.$$
\end{enumerate}
We will now characterize this situation in terms of elementary factors:

\begin{prop}\label{not21}
Let $M \in \Mat_n(\K)$ with no eigenvalue in $\{0,\alpha,\beta,\alpha+\beta\}$.
The following conditions are then equivalent:
\begin{enumerate}[(i)]
\item The elementary factors of $M$ are all polynomials of $(X-\alpha)\,(X-\beta)$.
\item The integer $n$ is even and there exists $N\in \Mat_{n/2}(\K)$ such that
$$M \sim \begin{bmatrix}
\alpha.I_{n/2} & N \\
I_{n/2} & \beta.I_{n/2}
\end{bmatrix}.$$
\item $M$ is an $(\alpha,\beta)$-composite.
\end{enumerate}
\end{prop}

\noindent We will start with a simple situation:

\begin{lemme}
Let $P \in \K[X]$ be a monic polynomial of degree $n \geq 1$, and set
$Y=(X-\alpha)\,(X-\beta)$.
Then
$$\begin{bmatrix}
\alpha.I_n & C(P) \\
I_n & \beta.I_n
\end{bmatrix} \, \sim \,C\bigl(P(Y)\bigr).$$
\end{lemme}

\begin{proof}
Setting $M:=\begin{bmatrix}
\alpha\, I_n & C(P) \\
I_n & \beta.I_n
\end{bmatrix}$, it will suffice to prove that
$P(Y)$, which has degree $2n$, is the minimal polynomial of $M$.
Simple computation shows that
$$(M-\alpha.I_n)\,(M-\beta.I_n)=\begin{bmatrix}
C(P) & 0 \\
0 & C(P)
\end{bmatrix},$$ which proves that
$P(Y)$ is an annihilator polynomial of $M$.  \\
Conversely, let $Q\in \K[X]$ be an annihilator polynomial of $M$.
The sequence
$$\bigl(1,X-\alpha,(X-\alpha)(X-\beta),\dots,(X-\alpha)^k(X-\beta)^k,(X-\alpha)^{k+1}(X-\beta)^k,\dots\bigr)$$
is clearly a basis of $\K[X]$, so we may split
$$Q=Q_1(Y)+(X-\alpha)\,Q_2(Y)$$
for some polynomials $Q_1$ and $Q_2$ in $\K[X]$.
Hence
\begin{align*}
Q(M)& =\begin{bmatrix}
Q_1\bigl(C(P)\bigr) & 0 \\
0 & Q_1\bigl(C(P)\bigr)
\end{bmatrix}
+\begin{bmatrix}
0 & C(P) \\
I_n & (\beta-\alpha).I_n
\end{bmatrix}\times \begin{bmatrix}
Q_2\bigl(C(P)\bigr) & 0 \\
0 & Q_2\bigl(C(P)\bigr)
\end{bmatrix} \\
& =\begin{bmatrix}
Q_1\bigl(C(P)\bigr) & ? \\
Q_2\bigl(C(P)\bigr) & ?
\end{bmatrix}.
\end{align*}
Since $Q(M)=0$, we deduce that $P$ divides $Q_1$ and $Q_2$, so $Q$ is a multiple of $P(Y)$. This proves that
$P(Y)$ is the minimal polynomial of $M$.
\end{proof}

\begin{proof}[Proof of Proposition \ref{not21}]
We have already proven that (ii) is equivalent to (iii).
For $A \in \Mat_m(\K)$, set
$$\varphi(A):=\begin{bmatrix}
\alpha.I_m & A \\
I_m & \beta.I_m
\end{bmatrix}.$$
\begin{itemize}
\item Assume (i) holds, and let $P_1,\dots,P_N$ denote the elementary factors of $M$.
For $k \in \lcro 1,N\rcro$, write $P_k=Q_k((X-\alpha)(X-\beta))$ for some $Q_k \in \K[X]$.
Hence
$$M \sim D\bigl(C(P_1),\dots,C(P_N)\bigr)$$
and, for every $k \in \lcro 1,N\rcro$, the companion matrix $C(P_k) \sim \varphi(C(Q_k))$ is an $(\alpha,\beta)$-composite,
so $M$ is an $(\alpha,\beta)$-composite, which in turn proves (ii).
\item Assume (ii) holds, and let $A \in \Mat_{n/2}(\K)$ such that $\varphi(A) \sim M$.
Let $Q_1,\dots,Q_N$ denote the elementary factors of $A$, so
$A \sim D\bigl(C(Q_1),\dots,C(Q_N)\bigr)$.
Set $P_k:=Q_k((X-\alpha)(X-\beta))$ for $k \in \lcro 1,N\rcro$.
A simple permutation of the basis shows then that
$$M \sim \varphi(A) \sim D\bigl(\varphi(C(Q_1)),\dots,\varphi(C(Q_n))\bigr)
\sim  D\bigl(C(P_1)),\dots,C(P_n)\bigr).$$
If $P_i$ divides $P_{i+1}$ for every suitable $i$, the $P_k$'s are the elementary factors of $M$,
which proves (i).
\end{itemize}
\end{proof}

\section{When all eigenvalues belongs to $\{0,\alpha,\beta,\alpha+\beta\}$}

Recall first Proposition 1 of \cite{HP}, the proof of which holds
regardless of the field $\K$:

\begin{prop}
Any nilpotent matrix is a difference of two idempotents.
\end{prop}

\noindent From this, we easily derive:

\begin{prop}\label{nilpotent}
Every nilpotent matrix is an $(\alpha,-\alpha)$-composite.
\end{prop}

\noindent The next proposition will be the last key to our theorems:

\begin{prop}\label{alphabeta}
Let $M \in \Mat_n(\K)$ be a triangularizable matrix with all eigenvalues in $\{\alpha,\beta\}$. \\
Assume $\alpha \neq \beta$. The following conditions are then equivalent:
\begin{enumerate}[(i)]
\item $M$ is an $(\alpha,\beta)$-composite;
\item The sequences $(n_k(M,\alpha))_{k \geq 1}$ and $(n_k(M,\beta))_{k \geq 1}$ are intertwined.
\end{enumerate}
\end{prop}

\noindent By Remark \ref{elremark}.(iii), this proposition has the following corollary:

\begin{cor}\label{0alpha+beta}
Assume $\alpha +\beta \neq 0$, and let $M \in \Mat_n(\K)$ denote a
triangularizable matrix with all eigenvalues in $\{0,\alpha+\beta\}$.
The following conditions are then equivalent:
\begin{enumerate}[(i)]
\item $M$ is an $(\alpha,\beta)$-composite;
\item The sequences $(n_k(M,0))_{k \geq 1}$ and $(n_k(M,\alpha+\beta))_{k \geq 1}$ are intertwined.
\end{enumerate}
\end{cor}

\noindent Assuming temporarily that Proposition \ref{alphabeta} holds, we can then prove the theorems with even numbers listed in section 1.
\begin{itemize}
\item Assume $\car(\K) \neq 2$ and $\alpha \neq \pm \beta$.
Then Theorem \ref{dSP2LCendo} follows directly from Propositions \ref{redprinc}, \ref{alphabeta} and \ref{0alpha+beta}.
\item Assume $\car(\K) \neq 2$ and $\beta=-\alpha$.
Notice that the polynomials of $(X-\alpha)\,(X+\alpha)=X^2-\alpha^2$ are simply the even polynomials. \\
The ``only if" part of Theorem \label{HPdiffendo}
then follows from Propositions \ref{redprinc}, \ref{not21} and \ref{alphabeta}.
For the ``if" part, we use the same results in conjunction with Proposition \ref{nilpotent}.

\item Assume $\car(\K)=2$ and $\beta=\alpha$.
The ``only if" part of Theorem \ref{HPsumcar2endo} then follows from Propositions \ref{redprinc} and \ref{not21}.
For the ``if" part, we use the same results in conjunction with
Proposition \ref{nilpotent} and the fact that for every nilpotent matrix $N$,
the matrix $\alpha.I_n+N$ is an $(\alpha,\alpha)$-composite since $N$ is an $(\alpha,-\alpha)$ composite.
\end{itemize}

\noindent It now only remains to prove Proposition \ref{alphabeta}: this will be done in the last section.

\section{Proof of Proposition \ref{alphabeta}}

\noindent Our proof will differ from that of
Hartwig and Putcha in \cite{HP}. More precisely, we will not rely upon the results of Flanders featured in \cite{Flanders},
but will try instead to prove the equivalence by elementary means.
We will need a few notations first.

\begin{Not}
When $p,q,r,s$ denote non-negative integers such that
$p \geq r$ and $q \geq s$, we set
$$K_{p,q}:=\begin{bmatrix}
\alpha.I_p & 0 \\
0 & \beta.I_q
\end{bmatrix} \in \Mat_{p+q}(\K) \quad \text{and} \quad
J_{p,q,r,s}:=\begin{bmatrix}
I_r & 0_{r,s}  \\
0_{p-r,r} & 0_{p-r,s} \\
0_{s,r} & -I_s \\
0_{q-s,r} & 0_{q-s,s}
\end{bmatrix} \in \Mat_{p+q,r+s}(\K).$$
\end{Not}

\noindent For the entire proof, we set a triangularizable matrix $M$
with all eigenvalues in $\{\alpha,\beta\}$.
We will simply write $n_k:=n_k(M,\alpha)$ and $m_k:=n_k(M,\beta)$ for $k \in \N^*$.

\subsection{Proof that (i) implies (ii)}

Assume that $M=\alpha.P+\beta.Q$ for some idempotents $P$ and $Q$.
The Jordan reduction theorem shows, after permuting the basis vectors,
that the matrix $M$ is similar to some block-triangular matrix
$$M'=\begin{bmatrix}
K_{n_1,m_1} & J_{n_1,m_1,n_2,m_2} & 0 & \dots & 0 \\
0 & K_{n_2,m_2} & J_{n_2,m_2,n_3,m_3} & & 0 \\
0 & 0 & K_{n_3,m_3} & \ddots & \vdots \\
\vdots &  &  & \ddots & \\
0 & \dots & &  0 & K_{n_N,m_N}
\end{bmatrix}$$
where $N$ denotes the index of the nilpotent matrix $(M-\alpha.I)\,(M-\beta.I)$.
Since the problem is invariant under similarity, we may assume that $M=M'$.

Remark that the flag of linear subspaces
which gives the previous block-decomposition of $M$
consists precisely of the iterated kernels of $(M-\alpha.I)\,(M-\beta.I)$.
Since the matrices $P$ and $Q$ commute with $(M-\alpha.I)\,(M-\beta.I)$,
they stabilize these subspaces, which proves that $P$ and $Q$ themselves decompose as
block-triangular matrices:
$$P=\begin{bmatrix}
P_{1,1} & P_{1,2} & \dots & P_{1,N} \\
0 & P_{2,2} & \dots & P_{2,N} \\
\vdots &  & \ddots & \vdots \\
0 & 0 &  & P_{N,N} \\
\end{bmatrix} \quad \text{and} \quad Q=\begin{bmatrix}
Q_{1,1} & Q_{1,2} & \dots & Q_{1,N} \\
0 & Q_{2,2} & \dots & Q_{2,N} \\
\vdots &  & \ddots & \vdots \\
0 & 0 &  & Q_{N,N} \\
\end{bmatrix}.$$
It is then clear that, for every $k \in \lcro 1,N-1\rcro$, the matrices
$\begin{bmatrix}
P_{k,k} & P_{k,k+1} \\
0 & P_{k+1,k+1}
\end{bmatrix}$ and
$\begin{bmatrix}
Q_{k,k} & Q_{k,k+1} \\
0 & Q_{k+1,k+1}
\end{bmatrix}$ are idempotents, which in turn proves that the matrix
$$\begin{bmatrix}
K_{n_k,m_k} & J_{n_k,m_k,n_{k+1},m_{k+1}} \\
0 & K_{n_{k+1},m_{k+1}}
\end{bmatrix} \quad \text{is an $(\alpha,\beta)$-composite.}$$

\noindent That the sequences $(n_k)_{k \geq 1}$ and
$(m_k)_{k \geq 1}$ are intertwined can then be deduced from the following lemma:

\begin{lemme}[Intertwinement lemma]
Let $p,q,r,s$ be non-negative integers such that $p \geq r$ and $q \geq s$. \\
Assume the block matrix
$M=\begin{bmatrix}
K_{p,q} & J_{p,q,r,s} \\
0 & K_{r,s}
\end{bmatrix}$ is an $(\alpha,\beta)$-composite.
Then $q \geq r$ and $p \geq s$.
\end{lemme}
In order to prove this, we will extract two matrices $A_1$ and $A_2$ such that
$$r \leq \rg(A_1)+\rg(A_2) \leq q.$$

\begin{proof}
Set
$K_1:=K_{p,q}$, $K_2:=K_{r,s}$ and $K_3:=J_{p,q,r,s}$, so that $M=\begin{bmatrix}
K_1 & K_3 \\
0 & K_2
\end{bmatrix}$. We choose two idempotents $P$ and $Q$ such that $M=\alpha.P+\beta.Q$.
Remark foremost that
$$(M-\alpha.I_{p+q})\,(M-\beta.I_{p+q})=\begin{bmatrix}
0 & I' \\
0 & 0
\end{bmatrix}, \quad
\text{with} \; I'=\begin{bmatrix}
(\alpha-\beta).I_r & 0_{r,s}  \\
0_{p-r,r} & 0_{p-r,s} \\
0_{s,r} & (\alpha-\beta).I_s \\
0_{q-s,r} & 0_{q-s,s}
\end{bmatrix} \in \Mat_{p+q,r+s}(\K).$$
The commutation argument already used earlier proves that there are three matrices $A \in \Mat_{p+q}(\K)$, $B \in \Mat_{p+q,r+s}(\K)$
and $C \in \Mat_{r+s}(\K)$ such that
$$P=\begin{bmatrix}
A & B \\
0 & C
\end{bmatrix}.$$
The idempotent $Q$ also has a decomposition of this type.
Consequently, both $A$ and $\frac{1}{\beta}\,(K_1-\alpha\,A)$
are idempotents, so
$$\beta \,(K_1-\alpha\,A)=(K_1-\alpha\,A)^2=K_1^2-\alpha\,(AK_1+K_1A)+\alpha^2 A^2
=K_1^2-\alpha\,(AK_1+K_1A)+\alpha^2 A.$$
From the definition of $K_1$, it is clear that $K_1^2=(\alpha+\beta).K_1-\alpha\,\beta.I_{p+q}$, and we deduce that
$$\alpha.K_1-\alpha\,(AK_1+K_1A)+\alpha\,(\alpha+\beta).A=\alpha\,\beta.I_{p+q}.$$
From this identity and the fact that $\alpha(\alpha-\beta)\neq 0$, we derive that there are matrices
$A_1 \in \Mat_{q,p}(\K)$ and $A_2 \in \Mat_{p,q}(\K)$ such that
$A=\begin{bmatrix}
I_p & A_2 \\
A_1 & 0
\end{bmatrix}$. Identity $A^2=A$ then entails that $A_2\,A_1=0$,
hence
$$\rg A_1+\rg A_2 \leq q.$$
We will now try to prove that $r \leq \rg A_1+\rg A_2$. \\
Commutation of $P$ with $(M-\alpha.I_n)\,(M-\beta.I_n)$ yields
that there are matrices $D_1 \in \Mat_{s,r}(\K)$, $L_1 \in \Mat_{s,p-r}(\K)$,
$N_1 \in \Mat_{q-s,p-r}(\K)$,
$D_2 \in \Mat_{r,s}(\K)$, $L_2 \in \Mat_{r,q-s}(\K)$, and $N_2 \in \Mat_{p-r,q-s}(\K)$ such that
$$A_1=\begin{bmatrix}
D_1 & L_1 \\
0 & N_1
\end{bmatrix} \quad , \quad A_2=\begin{bmatrix}
D_2 & L_2 \\
0 & N_2
\end{bmatrix} \quad \text{and} \quad
C=\begin{bmatrix}
I_r & D_2 \\
D_1 & 0
\end{bmatrix}.$$
Using again the identity $P^2=P$, we obtain:
$$A\,B+B\,C=B.$$
Since $Q=\frac{1}{\beta}\,(M-\alpha.P)$ and $Q$ is also idempotent, the corresponding identity for $Q$ yields:
$$\frac{1}{\beta}\,(K_1-\alpha.A)\,\frac{1}{\beta}\,(K_3-\alpha.B)
+\frac{1}{\beta}\,(K_3-\alpha.B)\,\frac{1}{\beta}\,(K_2-\alpha.C)=\frac{1}{\beta}\,(K_3-\alpha.B),$$
therefore
$$\beta\,K_3=K_1K_3+K_3K_2-\alpha\,(K_1B+BK_2)-\alpha\,(K_3C+AK_3)+\alpha\,(\alpha+\beta)\,B.$$
Using a block-decomposition of $B$, a simple computation allows us to
deduce from the previous identity that
there are matrices $B_1 \in \Mat_{s,r}(\K)$, $C_1 \in \Mat_{q-s,r}(\K)$ and
$B_2 \in \Mat_{r,s}(\K)$ such that
$$B=\begin{bmatrix}
\frac{1}{\alpha}\,I_r & B_2 \\
0 &  ? \\
B_1 & ? \\
C_1 & ?
\end{bmatrix}.$$
Computation of the first $r \times r$ block in the identity $AB+BC=B$ then yields:
$$D_2\,B_1+B_2\,D_1+L_2\,C_1=\frac{1}{\alpha}\,I_{r.}$$
For every $X \in \Ker D_1$, one has $D_2\,B_1\,X+L_2\,C_1\,X=\frac{1}{\alpha}\,X$,
which proves that
$$\dim (\im D_2+\im L_2) \geq \dim \Ker D_1,$$
hence
$$\rg \begin{bmatrix}
D_2 & L_2
\end{bmatrix} \geq r-\rg(D_1).$$
It follows that
$$r \leq \rg(D_1)+\rg \begin{bmatrix}
D_2 & L_2
\end{bmatrix} \leq \rg(A_1)+\rg(A_2).$$
This finally proves $r \leq q$. By an argument of symmetry, one also has $s \leq p$.
\end{proof}

\subsection{Proof of (ii) $\Rightarrow$ (i)}

\noindent We start with three special cases:

\begin{prop}
Let $n \geq 1$. Then each of the three matrices
$$A:=\begin{bmatrix}
J_n(\alpha) & 0 \\
0 & J_n(\beta)
\end{bmatrix} \quad , \quad
B:=\begin{bmatrix}
J_n(\alpha) & 0 \\
0 & J_{n+1}(\beta)
\end{bmatrix} \quad
\text{and} \quad
B':=\begin{bmatrix}
J_{n+1}(\alpha) & 0 \\
0 & J_n(\beta)
\end{bmatrix}$$
is an $(\alpha,\beta)$-composite.
\end{prop}

\begin{proof}
\begin{itemize}
\item Since $A$ is similar to the companion matrix $C\bigl((X-\alpha)^n\,(X-\beta)^n\bigr)$,
Proposition \ref{not21} proves that it is an $(\alpha,\beta)$-composite.
\item We can decompose
$$B=\begin{bmatrix}
A & C \\
0 & \beta
\end{bmatrix}, \quad \text{where $C=\begin{bmatrix}
0 \\
\vdots \\
0 \\
1
\end{bmatrix}
\in \Mat_{2n,1}(\K)$.}$$
We have found two idempotents $P$ and $Q$ such that $A=\alpha.P+\beta.Q$.
More precisely, the proof of Proposition \ref{not21} even provides $P$ and $Q$ with the additional
constraint: $\im P \oplus \Ker Q=\K^{2n}$.
We can then find two column matrices $C_1$ and $C_2$ such that
$$C_1 \in \im P, \quad C_2 \in \Ker Q \quad \text{and}\quad C=\alpha.C_1+\beta.C_2.$$
The matrices
$$P_1:=\begin{bmatrix}
P & C_1 \\
0 & 0
\end{bmatrix} \quad \text{and} \quad
Q_1:=\begin{bmatrix}
Q & C_2 \\
0 & 1
\end{bmatrix}$$
are then idempotents and satisfy $B=\alpha.P_1+\beta.Q_1$.
\item A similar argument proves that $B'$ is an $(\alpha,\beta)$-composite.
\end{itemize}
\end{proof}

Let now $M \in \Mat_n(\K)$ as in Proposition \ref{alphabeta}, and assume the two sequences
$(n_k)_{k \geq 1}=(n_k(M,\alpha))_{k \geq 1}$ and $(m_k)_{k \geq 1}=(n_k(M,\beta))_{k \geq 1}$
are intertwined.
Let $N_\alpha$ and $N_\beta$ denote the respective nilpotency indices associated to
the restriction of $M$ to $\Ker(M-\alpha.I_n)^n$ and $\Ker(M-\beta.I_n)^n$.
That the sequences $(n_k)_{k \geq 1}$ and $(m_k)_{k \geq 1}$ are intertwined shows that
$-1 \leq N_\alpha-N_\beta \leq 1$.
If $N_\alpha=0$ or $N_\beta=0$, then $M=\beta.I_n$ or $M=\alpha.I_n$ so $M$ is clearly an $(\alpha,\beta)$-composite.
Assume now that $N_\alpha \geq 1$ and $N_\beta \geq 1$.
Whether $N_\beta=N_\alpha$, $N_\beta=N_\alpha+1$ or $N_\beta=N_\alpha-1$, there is some matrix $M'$
such that $M$ is similar to either
$$\begin{bmatrix}
M' & 0 & 0 \\
0 & J_{N_\alpha}(\alpha) & 0 \\
0 & 0 & J_{N_\alpha}(\beta)
\end{bmatrix} \;, \;
\begin{bmatrix}
M' & 0 & 0 \\
0 & J_{N_\alpha}(\alpha) & 0 \\
0 & 0 & J_{N_\alpha+1}(\beta)
\end{bmatrix} \; \text{or} \;
\begin{bmatrix}
M' & 0 & 0 \\
0 & J_{N_\alpha}(\alpha) & 0 \\
0 & 0 & J_{N_\alpha-1}(\beta)
\end{bmatrix}.$$
In any case, we are reduced to proving that $M'$ is an $(\alpha,\beta)$-composite,
which follows easily by induction since $M'$ has its eigenvalues in $\{\alpha,\beta\}$ and the sequences
$(n_k(M',\alpha))_{k \geq 1}$ and $(n_k(M',\beta))_{k \geq 1}$ are easily shown to be intertwined.
This finishes our proof of Proposition \ref{alphabeta},
and all the theorems claimed in section 1 then follow.

\end{document}